\documentclass[10pt]{amsart}
\usepackage{amssymb}

\theoremstyle{plain}
\newtheorem{theorem}{Theorem}[section]

\newtheorem{lemma}[theorem]{Lemma}

\theoremstyle{remark}
\newtheorem{rem}[theorem]{Remark}

\newcommand{\OO}{{\mathcal O}}

\newcommand{\Z}{{\mathbb Z}}
\newcommand{\Q}{{\mathbb Q}}
\newcommand{\C}{{\mathbb C}}
\newcommand{\R}{{\mathbb R}}
\newcommand{\K}{{\mathbb K}}

\numberwithin{equation}{section}

\begin{document}

\title[On $k$-generalized Fibonacci numbers]{On the $k$-generalized Fibonacci numbers with negative indices}
\subjclass[2010]{11A63, 52C22}
\keywords{Fibonacci numbers, $k$-generalized Fibonacci numbers, diophantine equations}

\author[Attila~Peth\H{o}]{Attila~Peth\H{o}}
\address{Department of Computer Science,\\ University of Debrecen,\\
H-4002 Debrecen, P.O. Box 400}
\email{petho.attila@unideb.hu}
\date{\today}

\begin{abstract}
In these notes we study the $k$-generalized Fibonacci sequences - $(F_n^{(k)})_{n\in \Z}$ - with positive and negative indices. Denote $T_k(x)$ its characteristic polynomial. Our most interesting finding is that if $k$ is even then the absolute value of the second real root of $T_k(x)$ is minimal among the roots. Combining this with a deep result of Bugeaud and Kaneko \cite{BK} we prove that there are only finitely many perfect powers in $(F_n^{(k)})_{n\in \Z}$, provided $k$ is even. Another consequence is that, if $k$ and $l$ denote even integers then the equation $F_m^{(k)} = \pm F_n^{(l)}$ has only finitely many effectively computable solutions in $(n,m)\in \Z^2$. In the case $k=l=4$ we establish all solutions of this equation.
\end{abstract}

\maketitle

\centerline{\it Dedicated to the $80^{th}$ birthday of K\'alm\'an Gy\H{o}ry}

\section{Introduction}

The Fibonacci sequence, which is defined by the initial terms $F_0=0, F_1=1$ and by the recursion $F_{n+2}=F_{n+1}+F_n$ is one of the most investigated sequences of integers. There are 36 books in the MathSciNet with the word "Fibonacci" in their title. The Fibonacci numbers have the analytic expression
$$
F_n = \frac{\alpha^n - \beta^n}{\alpha - \beta} \; \mbox{with} \; \alpha=\frac{1+\sqrt{5}}{2}, \; \mbox{and} \; \beta=\frac{1-\sqrt{5}}{2},
$$
which is called Bin\'et formula.

The Fibonacci sequence can be computed not only forward, but backward, too by using the recursion $F_n = -F_{n+1} + F_{n+2}$. Replacing here $n$ by $-n$ we get $F_{-n} = - F_{-(n-1)} + F_{-(n-2)}$, which can be written in the more familiar form $F_m= -F_{m-1} +F_{m-2}$. The numbers $\alpha$ and $\beta$ are the roots of the characteristic polynomial $x^2-x-1$ of the Fibonacci sequence. Thus they satisfy $\alpha \beta =-1$, i.e, $\beta = -\alpha^{-1}$ and $\alpha=-\beta^{-1}$. Inserting this in the Bin\'et formula we get
$$
F_{-n} = \frac{\alpha^{-n} - \beta^{-n}}{\alpha - \beta} = (-1)^n\frac{\alpha^{n} - \beta^{n}}{\alpha - \beta} = (-1)^n F_n,
$$
which shows strong relation between positive and negative indices Fibonacci numbers.

Among the many generalizations of the Fibonacci sequence we concentrate here to the $k$-generalized Fibonacci sequence\footnote{These are called in the literature by  multinacci sequence \cite{Arkin}, generalized Fibonacci sequence \cite{Wolfram}. There appeared recently many papers on diophantine properties calling our subject $k$-generalized Fibonacci sequence, see e.g. \cite{Dresden, Marques,Sanna}, hence we use it, too.}. Let $k\ge 2$, and define the sequence $(F_n^{(k)})$ by the initial values $F_n^{(k)}=0$ for $n=-k+2,\dots,0, F_1^{(k)}=1$ and by the $k^{th}$ order recursion $F_{n+k}^{(k)}= F_{n+k-1}^{(k)}+\dots+F_{n}^{(k)}, n\ge -k+2$. Of course for $k=2$ we get the Fibonacci sequence.

Denote by $T_k(x)$ the characteristic polynomial of $(F_n^{(k)})$, i.e., set $T_k(x)=x^k-x^{k-1}-\ldots-1$. Denote $\alpha_1,\dots,\alpha_k$ the roots of $T_k(x)$. Miles \cite{Miles} and later Miller \cite{Miller} proved that $T_k(x)$ has simple roots. Wolfram \cite{Wolfram} seems to be the first who dealt with the location of the roots. He proved that $T_k(x)$ is a Pisot polynomial, i.e. all but one roots lie inside the unit disc. More precisely he showed \cite[Lemma 3.6]{Wolfram} that $\alpha_1> 2(1-2^{-k})$, $T_k(x)$ `` has one negative real root in the interval $(-1,0)$ when $k$ is even. This root and each complex root $r$ has modulus $3^{-k} < |r |< 1$." It is well-known that the terms of linear recurrences can be expressed in the form $F_n^{(k)}= c_1 \alpha_1^n +\dots+c_k\alpha_k^n$ with $c_1,\dots,c_k$ suitable constants. Let $g_k(x) = \frac{x-1}{2+(k+1)(x-2)}$. In the actual case Dresden \cite{Dresden} proved that $c_l = \frac1{\alpha_l}g_k(\alpha_l)$. Moreover he proved the very sharp inequality
\begin{equation}\label{eq:Dresden1}
  |F_n^{(k)}- g_k(\alpha_1)\alpha_1^{n-1}| < \frac12.
\end{equation}
For simplicity we will take $C_l=g_k(\alpha_l), l=1,\dots,k$ and notice $C_l\not= 0$. Indeed, $C_l=0$ if and only iff $\alpha_l=1$, which is impossible. With this notation we get
\begin{equation}\label{eq:Binet}
  F_n^{(k)}= C_1 \alpha_1^{n-1} +\dots+C_k\alpha_k^{n-1}.
\end{equation}
Dresden, following Wolfram, proved this formula for $n\ge 0$.

Like the Fibonacci sequence, the $k$-generalized Fibonacci sequence too, can be continued to negative indices such that the terms are rational integers. This happens because the coefficient of $F_n^{(k)}$ in the recursion is one. Of course we have $F_{n}^{(k)} = -F_{n+1}^{(k)}-\dots-F_{n+k-1}^{(k)}+F_{n+k}^{(k)}$. Replacing here $n$ by $-n$ we get
$$
F_{-n}^{(k)} = -F_{-(n-1)}^{(k)}-\dots-F_{-(n-k+1)}^{(k)}+F_{-(n-k)}^{(k)}.
$$
Notice that Dresden's formula \eqref{eq:Binet} remains true for negative $n$'s too. To apply it, for example, to study diophantine properties of $F_{-n}^{(k)}$ we need more precise information on the roots of $T_k(x)$. We present some of them in these notes. After the elementary proof I searched the literature for similar results, and found only the paper of Ruiz and Luca \cite{RL}. Although it includes quite accurate estimation on the size of the roots of $T_k(x)$, it does not study the location of the second real root if $k$ is even.

Our most interesting finding is that if $k$ is even then the second real root of $T_k(x)$ has the least absolute value among the roots. Combining this with a deep result of Bugeaud and Kaneko \cite{BK} we prove (Theorem \ref{th:4-genppowers}) that there are only finitely many perfect powers in $(F_n^{(k)})_{n\in \Z}$, provided $k$ is even. Another consequence is that, if $k$ and $l$ denote even integers then the equation $F_m^{(k)} = \pm F_n^{(l)}$ has only finitely many effectively computable solutions in $(n,m)\in \Z^2$ (Theorem \ref{th:effective}). In the case $k=l=4$ we establish all solutions of this equation (Theorem \ref{th:4-gen}).

\section{On the roots of $T_k(x)$}

Our first result is that the roots of $T_k(x)$ are not only simple, but two roots may have same absolute value only if they are conjugate complex numbers. Our first theorem follows easily from a result of Mignotte \cite{Mig}, but our proof is elementary.

\begin{theorem}\label{th:egyenlo}
  Let $z_1,z_2$ be different roots of $T_k(x)$ such that $|z_1|= |z_2|$. Then $z_2=\bar{z_1}$.
\end{theorem}

\begin{proof}
Following Dresden \cite{Dresden} set $f_k(x)=T_k(x)(x-1)=x^{k+1}-2x^k +1$. Beside $1$, the roots of $f_k(x)$ are the same as of $T_k(x)$. The roots of $T_k(x)$ are simple.

   If $z_1,z_2$ are both real then $|z_1|= |z_2|$ implies $z_2=-z_1$. Assuming $z_1>0$ we have
  $$
  z_1^{k+1}-2z_1^k+1 = (-1)^{k+1}z_1^{k+1}-2(-1)^k z_1^k +1 =0,
  $$
  and after simple transformations we get
  $$
  z_1-2 = (-1)^{k+1}(z_1+2).
  $$
  This is impossible if $k$ is odd. If $k$ is even, then $z_1=2$, which is never a root of $T_k(x)$.

  \medskip

  In the sequel we assume that $z_1,z_2$ are non-real complex numbers. Denote $r$ their common absolute value, and $\omega_1, \omega_2$ their arguments, i.e., let $z_1=r(\cos\omega_1+i\sin \omega_1), z_2=r(\cos\omega_2+i\sin \omega_2)$. The equation $f_k(z_1)=f_k(z_2)=0$ implies
  $$
  r^{k+1}u_1-2r^{k}v_1+1
  = 0=
  r^{k+1}u_2-2r^{k}v_2+1.
  $$
  with
  \[\begin{array}{lclccl}
  u_1&=&\cos(k+1)\omega_1 + i\sin(k+1)\omega_1, &u_2&=& \cos(k+1)\omega_2 + i\sin(k+1)\omega_2,\\
  v_1&=& \cos k\omega_1 + i\sin k\omega_1, &v_2 &=&\cos k\omega_2 + i\sin k\omega_2.
  \end{array}
  \]
  After some elementary manipulation we obtain
  \begin{equation}\label{eq:trigon}
    r(u_1-u_2) = 2(v_1-v_2).
  \end{equation}
  The argument of the complex number $u_1-u_2$ is $\frac{\pi}{2}+\frac{(k+1)\omega_1+(k+1)\omega_2}{2}$. Denote its absolute value by $s$. Similarly, the argument of $v_1-v_2$ is $\frac{\pi}{2}+\frac{k\omega_1+k\omega_2}{2}$. Denote its absolute value by $t$. We have $s=0$ if and only if $u_1= u_2$ and $t=0$. Then $v_1= v_2$ holds as well. These identities imply $e^{i\omega_1} = e^{i\omega_2} $ by using Euler's identity $\cos\alpha+i\sin\alpha = e^{i\alpha}$. Hence $\omega_1=\omega_2$, i.e.,  $z_1=z_2$.

  In the sequel we may assume $st\not=0$. Now \eqref{eq:trigon} implies
  $$
  rs \cdot e^{-i (k+1) \frac{\omega_1+\omega_2}{2}}  = 2t \cdot e^{-i k\frac{\omega_1+\omega_2}{2}}\quad \leftrightarrow \quad rs \cdot e^{-i \frac{\omega_1+\omega_2}{2}}  = 2t .
  $$

  Hence $\omega_2 = -\omega_1$, i.e. $z_2= \bar{z_1}$, as stated.
\end{proof}

\begin{rem}
  By our first theorem, except the conjugate complex pairs, the absolute values of the roots of $T_k(x)$ are different. If $k\not=l$, then the polynomials $T_k(x)$ and $T_l(x)$ have no common roots. We conjecture the stronger statement, that if $k\not=l$, then the absolute values of the roots of $T_k(x)$ and $T_l(x)$ are different.
\end{rem}

Let $k$ be odd. Then $T_k(x)$ has one real root $\alpha_1$ and $\frac{k-1}{2}$ pairs of conjugate complex roots, whose modulus are by Theorem \ref{th:egyenlo} different.
%We will prove later
%$$
%|T_n^{(k)}| = |g_k(\alpha_{k-1}) \alpha_{k-1}^{n-1} + g_k(\overline{\alpha_{k-1}}) \overline{\alpha_{k-1}}^{n-1}| %+ o(|\alpha_{k-1}^{n-1}|),
%$$
%provided $n\le 0$.
What happens, if $k$ is even, when $T_k(x)$ has two real roots? The second real root lie by Wolfram \cite{Wolfram} in the interval $(-1,-3^{-k})$, but what is its comparison to the complex roots? Our next result shows that the smallest in modulus root is the second real root. To prove this we use Rouch\'e's theorem in the following form, see e.g. Filaseta \cite{Filaseta}.

\begin{theorem} \label{th:Rouche}
  Let $f(z)$ and $g(z)$ be polynomials with complex coefficients and $C=\{z\; : \; |z|=1\}$. If the strict inequality
  $$
  |f(z)+g(z)| < |f(z)| + |g(z)|
  $$
  holds at each point on the circle $C$, then $f(z)$ and $g(z)$ must have the same total number of roots (counting multiplicity) strictly inside $C$.
\end{theorem}

In the sequel $\alpha_{k1},\ldots,\alpha_{kk}$ will denote the roots of $T_k(x)$, and we assume that they are ordered as $|\alpha_{k1}|\ge \ldots,\ge |\alpha_{kk}|$. If, however, $k$ is in some sense fixed, for example we are dealing with one $T_k(x)$, then the first index will be omitted.

\begin{theorem} \label{th:smallest}
  Let $\alpha_1,\dots,\alpha_{2k}$ be the roots of $T_{2k}(x)$. Then $\alpha_{2k}$ is real and $|\alpha_{2k}|< |\alpha_{j}|$ for $j=1,\dots,2k-1$.
\end{theorem}

\begin{proof}
  Let
  $$
  g_{2k}(x) = x^{2k+1}f_{2k}\left(-\frac{1}{x}\right) = x^{2k+1}-2x-1,
  $$
  where $f_k(x)$ denotes the polynomial introduced in the proof of Theorem \ref{th:egyenlo}. As $g_{2k}(1)=-2$ and $g_{2k}(2)\ge 3$ the polynomial $g_{2k}(x)$ has a real root in $(1,2)$, which we denote by $\lambda$. As $g_{2k}(x)$ is monotone increasing on $(1,\infty)$ the number $\lambda$ is its only real root in this interval. Let $\mu>\lambda$, then $g_{2k}(\mu)>0$. If $z\in \C$ is such that $|z|= \mu$, then
  $$
  |2z + 1| \le 2|z| + 1 = 2\mu +1 < \mu^{2k+1}.
  $$
  Since $x^{2k+1}$ has $2k+1$ roots inside the disc $|z| <\mu$, it follows from Theorem \ref{th:Rouche}, that $g_{2k}(x)$ has the same number of roots inside this disc, i.e., the absolute value of all roots of $g_{2k}(x)$ is at most $\mu$. As $\mu>\lambda$ is arbitrary all roots lie in the closed disc $|z|\le \lambda$. By Theorem \ref{th:egyenlo} $\lambda$ is the only root of $g_{2k}(x)$ with absolute value $\lambda$, thus the absolute value all other roots must be smaller, i.e. $\lambda = -1/\alpha_{2k}$.

\end{proof}

\begin{theorem} \label{th:sorrend}
  Let $\alpha_{k1},\dots,\alpha_{kk}$ be the roots of $~T_k(x)$ ordered as $|\alpha_{k1}|\ge \dots \ge |\alpha_{kk}|$. Then
  \begin{enumerate}
    \item if $0<k <l$ then $\alpha_{2l,2l} < \alpha_{2k,2k}$,
    \item if $0<k <l$ then $\alpha_{l,1} > \alpha_{k,1}$,
    \item if $l\ge 1$ and $k\ge 2$, then  $\alpha_{k1}\ge -\frac{1}{\alpha_{2l,2l}}$, and equality holds if and only if $l=1,k=2$.
  \end{enumerate}
\end{theorem}

\begin{proof}
  (1) Set $\xi_l=-1/\alpha_{2l,2l}$, and $\xi_k=-1/\alpha_{2k,2k}$. Then $\xi_k,\xi_l>1$ are roots of $g_{2l}(x)$, and $g_{2k}(x)$ respectively, where $g_{2k}(x)$ denotes the polynomial introduced in the proof of Theorem \ref{th:smallest}. The identities $g_{2k}(\xi_k)=g_{2l}(\xi_l)=0$ imply
  $$
  \xi_k^{2k+1} - \xi_l^{2l+1} = 2(\xi_k -\xi_l),
  $$
  which we can rearrange as
  $$
  \xi_k^{2k+1} - \xi_l^{2k+1} - 2(\xi_k -\xi_l)   = \xi_l^{2k+1}(\xi_l^{2(l-k)}-1).
  $$
  After division by $\xi_k -\xi_l$, which is non-zero because $\alpha_{2l,2l}$ and $\alpha_{2l,2l}$ are roots of the irreducible polynomials $T_l(x)$ and $T_k(x)$ respectively, we obtain
  $$
  \sum_{j=1}^{2k}\xi_k^j \xi_l^{2k-j} -2 = \frac{\xi_l^{2k+1}(\xi_l^{2(l-k)}-1)}{\xi_k -\xi_l}.
  $$
  As $\xi_k,\xi_l>1$ the left hand side and the numerator of the right hand side are positive, hence the equality is only possible if $\xi_k> \xi_l$, i.e., if $\alpha_{2l,2l} < \alpha_{2k,2k}$.

  (2) Using that $\alpha_{k1}>1$ is a root of $f_k(x)$ we can prove this statement by the same argument as (1).

(3) By (1), as a function in $k$, $\xi_k = -1/\alpha_{2k,2k}$ is strongly monotone decreasing, while by (2) $\alpha_{k1}$ is strongly monotone increasing. We finish the proof with the equality $\xi_2=\frac{1+\sqrt{5}}{2}= \alpha_{21}$.

\end{proof}

\section{Ineffective results}

\subsection{Perfect powers} In this section we start the investigation of the diophantine properties of the $k$-generalized Fibonacci sequences. We pay special attention their members with negative indices. We start with perfect powers, which is one of my favourite topics. In a recent paper Bugeaud and Kaneko \cite{BK} proved

\begin{theorem}\label{t:Bugeaud_Kaneko}
  Let $(u_n)_{n\ge 0}$ be a linear recurrence sequence of integers of order at least two
  and such that its characteristic polynomial is irreducible and has a dominant root. Then there
are only finitely many perfect powers in $(u_n)_{n\ge 0}$. Moreover, their number can be bounded by
an effectively computable constant.
\end{theorem}

A consequence of it is

\begin{theorem} \label{th:4-genppowers}
  If $k>0$ is even then there are only finitely many perfect powers in the (two sided) sequence $(F_n^{(k)})_{n\in \Z}$.
  Moreover, their number can be bounded by an effectively computable number.
\end{theorem}

\begin{proof}
  Spit $(F_n^{(k)})_{n\in \Z}$ into the union of two sequences $(F_n^{(k)})_{n\ge 0}$ and $(F_{-n}^{(k)})_{n> 0}$.
  Their characteristic polynomials are $T_k(x)$ and $x^kT_k(1/x)$. Wolfram \cite{Wolfram} proved that $T_k(x)$ is irreducible, thus
  $x^kT_k(1/x)$ is irreducible too. Again by Wolfram \cite{Wolfram} $\alpha_{1}$ is the dominating root of $T_k(x)$.
  If $k$ is even then by Theorem \ref{th:sorrend} $1/\alpha_{k}$ is the dominating root of $x^kT_k(1/x)$. Hence the
  assumptions of Theorem \ref{t:Bugeaud_Kaneko} hold for both sequences $(F_n^{(k)})_{n\ge 0}$ and $(F_{-n}^{(k)})_{n> 0}$. Thus
  there are only finitely many perfect powers in both sequences and in their union too.
\end{proof}

Bugeaud, Mignotte and Siksek \cite{BMS} established all perfect powers in the Fibonacci sequence, i.e. we know all solutions of the equation $F_n^{(2)}=x^q,\; n,x,q\in \Z,\; q\ge 2$. For $k>2$ already the ineffective Theorem \ref{t:Bugeaud_Kaneko} of Bugeaud and Kaneko is a big breakthrough. If $k$ is odd then $(F_n^{(k)})_{n\ge 0}$ satisfies the assumptions of Theorem \ref{t:Bugeaud_Kaneko}, but the irreducible polynomial $x^kT_k(1/x)$ has two, conjugate complex roots with maximal absolute value, hence this theorem is not applicable. By our opinion $(F_{-n}^{(k)})_{n> 0}$ has for $k$ odd  finitely many perfect powers too.

\medskip
\subsection{Common terms} In the sequel we concentrate on common terms of the $k$-generalized Fibonacci sequences. First we prove ineffective results, where the basic tool is the theory of $S$-unit equations. We define them here and cite the fundamental theorem on such equations. For an algebraic number field $\K$ denote $M_{\K}$ its set of places. Let $S\subset M_{\K}$ be finite including all archimedean places, let $\OO_{S}$ denote the set of {\it $S$-integers} of $\K$, i.e., the set of those elements $\alpha \in \K$ with $|\alpha|_v \le 1$ for all $v\in M_{\K}\setminus S$.

Consider the {\it weighted $S$-unit equation}
\begin{equation}\label{Sunit}
  \alpha_1 X_1+\dots+ \alpha_s X_s =1,
\end{equation}
where $s\ge 2$, $\alpha_1,\dots,\alpha_s$ are non-zero elements of $\K$ and the solutions $x_1,\dots,x_s$ belong to $\OO_S$. A solution $x_1,\dots,x_s$ of \eqref{Sunit} is called {\it degenerate} if there exists a proper subset $I$ of $\{1,\dots,s$\} such that $\sum_{i\in I}\alpha_ix_i = 0$. The next theorem was proved by Evertse \cite{Evertse} and independently by van der Poorten and Schlickewei \cite{PoortenSch}, see also \cite{Evertse_Gyory}.

\begin{theorem} \label{th:Sunit}
  Equation \eqref{Sunit} has only finitely many non-degenerate solutions in $x_1,\dots,x_s\in \OO_S$.
\end{theorem}

Consider the diophantine equation
\begin{equation}\label{FMk=Fnl}
  F_m^{(k)} = F_n^{(l)}
\end{equation}
in integers $k,l>0$ and $n,m$. If $k$ and $l$ are fixed then \eqref{FMk=Fnl} has only finitely many effectively computable solutions in $n,m\ge 0$, see Mignotte \cite{Mignotte} and also Kiss \cite{KissP1,KissP2}.

Marques \cite{Marques} proved that if $l> k\ge 2, n>l+1$ and $m>k+1$ then \eqref{FMk=Fnl} has only the solutions
$$
(m,n,l,k) = (7,6,3,2) \; \mbox{and} \; (12,11,7,3).
$$
This result describes completely the intersection of two sets of $k$-generalized Fibonacci numbers with non-negative indexes because $F_1^{(k)}=1$, and $F_m^{(k)}=2^{k-1}$ for $k\ge 2, 1\le m\le k$. Much less is known if $n$ or $m$ is negative. In this direction the first step was done by Bravo, G\'omez and Luca \cite{BGL}, see also Bravo et al. \cite{BGLTK}, who solved completely the equation $F_m^{(3)} = F_n^{(3)}$ in integers $m,n$. Their proof depend basically on the relation $|\alpha_{3,2}| = |\alpha_{3,3}|= \sqrt{\alpha_{3,1}}$. It seems, unfortunately, that there is no similar simple algebraic relation between the roots of $T_k(x)$, provided $k>3$. However our investigations on the roots of $T_k(x)$ allows us to extend the results of Bravo et al. for $k>3$, although in a weaker ineffective form. Combining our Theorem \ref{th:egyenlo} with Theorem \ref{th:Sunit} we prove

\begin{theorem} \label{th:Fmk=Fnl}
  Let $l, k\ge 2$ be fixed. Then the diophantine equation \eqref{FMk=Fnl} has only finitely many solutions $(m,n)\in \Z^2$.
\end{theorem}

Our Theorem \ref{th:Fmk=Fnl} is much weaker than Marques's above cited result. The reason is that for negative indices the dominating root is not large enough, and all but one roots are lying outside the unit circle.

\begin{proof}
    Set $C_{ks}=g_k(\alpha_{ks}), s=1,\dots,k$. With this notation, using \eqref{eq:Binet}, equation \eqref{FMk=Fnl} can be written in the form
  \begin{equation*}
    C_{k1}\alpha_{k1}^{m-1}+\ldots+ C_{kk}\alpha_{kk}^{m-1} = C_{l1}\alpha_{l1}^{n-1}+\ldots+ C_{ll}\alpha_{ll}^{n-1}.
  \end{equation*}
 Let $\K=\Q(\alpha_{k1},\dots,\alpha_{kk},\alpha_{l1},\dots,\alpha_{ll})$ and $\OO$ the ring of integers of $\K$, finally $S=\emptyset$. Dividing the last equation by $C_{l1}\alpha_{l1}^{n-1}$, which is obviously non-zero we get the relation
 \begin{equation*}
 a_1\frac{\alpha_{k1}^{m-1}}{\alpha_{l1}^{n-1}}+\dots+a_k\frac{\alpha_{kk}^{m-1}}{\alpha_{l1}^{n-1}}+a_{k+1} \left(\frac{\alpha_{l2}}{\alpha_{l1}}\right)^{n-1}+\dots+ a_{k+l-1} \left(\frac{\alpha_{ll}}{\alpha_{l1}}\right)^{n-1} =1,
  \end{equation*}
  where
  \[
  a_s = \left\{ \begin{array}{cl}
    \frac{C_{ks}}{C_{l1}}, & s=1,\dots,k \\[1ex]
    -\frac{C_{l,s-k+1}}{C_{l1}}, & s=k+1,\dots,k+l-1.
  \end{array} \right.
  \]
  The elements $a_s, s=1,\dots,k+l-1$ are non-zero and belong to $\K$. Further, as the roots of $T_k(x), T_l(x)$ are units in $\K$, the same do the elements $\frac{\alpha_{ks}^{m-1}}{\alpha_{l1}^{n-1}}, s=1,\dots,k$ and $\left(\frac{\alpha_{ls}}{\alpha_{l1}}\right)^{n-1}, s=2,\dots,l$.

If our equation \eqref{FMk=Fnl} has infinitely many solutions $m,n\in \Z$ then the equation
\begin{equation}\label{eq:unit}
  a_1x_1+\dots+a_{k+l-1} x_{k+l-1} = 1
\end{equation}
  has infinitely many $S$-unit solutions belonging to $\OO$. As $S=\emptyset$ the $S$-units of $\OO$ are the same as its units. By Theorem \ref{th:Sunit} all but finitely many solutions of our equation are degenerate, i.e., satisfy an equation $\sum_{j\in J} a_j x_j = 0$ with $\emptyset \not=J\subset \{1,\dots,k+l-1\}$.

  {\bf Case I.} If $\emptyset \not= \bar{J}\subseteq \{k+1,\dots,k+l-1\}$ then after repeated application of Theorem \ref{th:Sunit} we get that there are $1\le j< h\le l$ such that $\alpha_{lj}/\alpha_{lh}$ is a root of unity. Let $\pi$ be a Galois conjugation of $\K$, which maps $\alpha_{lj}$ to $\alpha_{l1}$. Then $\pi(\alpha_{lh})=\alpha_{lt}$ with some $2\le t\le l$ and $\pi(\alpha_{lj}/\alpha_{lh})=\alpha_{l1}/\alpha_{lt}$ is a root of unity. However this is impossible by Wolfram's result \cite{Wolfram}.
  \smallskip

  {\bf Case II.} If $\emptyset \not= \bar{J}\subseteq \{1,\dots,k\}$ then after repeated application of Theorem \ref{th:Sunit} we get that there are $1\le j< h\le k$ such that $\alpha_{kj}/\alpha_{kh}$ is a root of unity. From here on repeat the argument of Case I.

  If neither Case I nor Case II appears then the equation
  $$
  \sum_{j\in \bar{J}} a_j x_j =1
  $$
  has infinitely many solutions among the units of $\OO$. This equation has the same shape than \eqref{eq:unit}, hence repeating the argument we arrive after some stages either at Case I or at Case II, which completes the proof.
\end{proof}

\section{Effective results}

Our first lemma provides a similar lower bound as \eqref{eq:Dresden1} for the growth of $(F_n^{(k)})$ if $k$ is even and $n<0$. Although it is much weaker, it is still good enough not only to prove effective finiteness results, but also to solve completely diophantine equations related to such sequences.

\begin{lemma} \label{l:domineq}
  If $k>2$ be even and $n<0$ then
  \begin{equation}\label{eq:domineq}
    |F_n^{(k)}-g_k(\alpha_{kk}) \alpha_{kk}^{n-1}| \le c_1 |\alpha_{kk}^{\delta n}|,
  \end{equation}
  where $c_1= |g_k(\alpha_{k1})|+\sum_{j=2}^{k-1}\frac{|g_k(\alpha_{kj})|}{|\alpha_{kj}|}$ and $\delta=\frac{\log |\alpha_{k,k-1}|}{\log|\alpha_{kk}|}<1$.
\end{lemma}

\begin{proof}
  The formula of Dresden \eqref{eq:Binet} implies
  $$
  F_n^{(k)} - g_k(\alpha_{kk}) \alpha_{kk}^{n-1} = \sum_{j=1}^{k-1} g_k(\alpha_{kj}) \alpha_{kj}^{n-1}.
  $$
  By the definition of $\delta$ we have $|\alpha_{k,k-1}|=|\alpha_{kk}|^{\delta}$, thus
  $$
  |\alpha_{kk}|^{n\delta} = |\alpha_{k,k-1}|^n= |\alpha_{k,k-2}|^n > |\alpha_{kj}|^n, \; j=1,\ldots,k-3.
  $$
  This estimate together with the formula of Dresden proves the statement.
\end{proof}

The following theorem is a simple consequence of Theorem \ref{th:smallest} and the Th\'eor\`{e}me of Mignotte \cite{Mignotte}.

\begin{theorem} \label{th:effective}
  Let the integers $k,l\ge 2$ be given. If $(n,m)\in \Z^2$ is a solution of the equation
  \begin{equation}\label{FMk=pmFnl}
    F_m^{(k)} = \pm F_n^{(l)},
  \end{equation}
  then there exists an effectively computable constant $C$ depending only on $k,l$ and the roots of $T_k(x), T_l(x)$ such that
  \begin{enumerate}
    \item $|n|,|m|<C$ provided $k$ and $l$ are even,
    \item $0\le n,|m|<C$ provided $k$ is odd and $l$ is even.
  \end{enumerate}
 \end{theorem}

A consequence of Lemma \ref{l:domineq} is that for even $k$ and small enough $n$ the consecutive terms of $(F_{-n}^{(k)})_{n=0}^{\infty}$ have opposite signs. Hence allowing $k$-generalized Fibonacci numbers with negative indices it is more natural (and general) to consider \eqref{FMk=pmFnl} instead of \eqref{FMk=Fnl}.

\begin{proof}
  We detail only the proof of (1), because the proof of (2) is similar, even simpler. We distinguish three cases according the signs of $n$ and $m$.

  {\bf Case i: $m,n\ge 0$.} For $k\not= l$ the statement, even in much stronger form, was proved by Marques \cite{Marques}. If $k=l$ then as $F_1^{(k)}, F_2^{(k)}=1$ and $(F_n^{(k)})_{n=2}^{\infty}$ is strict monotone increasing, the assertion is obviously true.

  {\bf Case ii: $m\ge 0, n< 0$.} Let $\alpha_{k1},\ldots,\alpha_{kk}$, and $\alpha_{l1},\ldots,\alpha_{ll}$ be the roots of $T_k(x)$ and $T_l(x)$ respectively. Order them as in Theorem \ref{th:sorrend}. Then both $(F_m^{(k)})_{m=0}^{\infty}$ and $(\pm F_{-n}^{(l)})_{n=1}^{\infty}$ are linear recursive sequences with the dominating terms $g_k(\alpha_{k1})\alpha_{k1}^{m-1}$ and $\pm g_l(\alpha_{ll})\alpha_{ll}^{n-1}$. A simple adaptation of the argument of Marques \cite{Marques}, p.460 shows that $\alpha_{k1}$ and $\alpha_{ll}$ are multiplicatively independent. Thus by the Th\'eor\`{e}me of Mignotte \cite{Mignotte} the equation $\eqref{FMk=Fnl}$ has only finitely many effectively computable solutions.

 {\bf Case iii: $m,n\le 0$.} Using the notation of Case ii we see that the sequences $(F_{-m}^{(k)})_{m=1}^{\infty}$ and $(\pm F_{-n}^{(l)})_{n=1}^{\infty}$ are linear recursive sequences with the dominating terms $g_k(\alpha_{kk})\alpha_{kk}^{m-1}$ and $\pm g_l(\alpha_{ll})\alpha_{ll}^{n-1}$. The numbers $\alpha_{kk}$ and $\alpha_{ll}$ are again multiplicatively independent, which allows us to use Mignotte's result.
\end{proof}

G\'omez and Luca \cite{BGL} established all solutions of $T_m^{(3)} = 0, m\in \Z$, i.e., all zero terms in the Tribonacci sequence. Bravo et al. \cite{BGLTK} proved that there are only eight integers, which appear at least twice in the Tribonacci sequence, and computed all solutions of $T_m^{(3)} = c$ in the remaining eight cases. Here we prove

\begin{theorem}\label{th:mult}
  Let $k\ge 2$ and $c\in \Z$. Then the equation
  \begin{equation}\label{eq:mult}
    F_m^{(k)} = c
  \end{equation}
  has only finitely many effectively computable solutions $m\in \Z$. If $c$ is large enough then \eqref{eq:mult} has at most one solution.
\end{theorem}

\begin{proof}
  For $m\ge 0$ the assertion follows immediately from \eqref{eq:Dresden1}. In the sequel we assume $m<0$, and distinguish two cases according the parity of $k$. We order the roots of $T_k(x)$ as $\alpha_1>|\alpha_2|\ge \dots \ge |\alpha_k|$.

  {\bf Case I: $k$ even.} Then, by Theorem \ref{th:smallest} $|\alpha_k|< |\alpha_j|, j=1,\dots,k-1$, i.e., $|\alpha_k|^{-1}> |\alpha_j|^{-1}, j=1,\dots,k-1$. Hence by Lemma \ref{l:domineq} we have
  $$
  |g_k(\alpha_k)\alpha_k^m| - c_1|\alpha_k^{\delta m}| \le |F_m^{(k)}| \le |g_k(\alpha_k)\alpha_k^m| + c_1|\alpha_k^{\delta m}|
  $$
  with simply computable constants $c_1>0$, and $\delta=\frac{\log |\alpha_{k-1}|}{\log|\alpha_{k}|}<1$. The assertions follow.

  {\bf Case II: $k$ odd.} In this case $\alpha_k$ is a non-real complex number, i.e., $\alpha_{k-1}= \bar{\alpha}_k$, but by Theorem \ref{th:egyenlo} $|\alpha_k|< |\alpha_j|, j=1,\dots,k-2$. As we proved above $\alpha_{k-1}/\alpha_{k}$ is not a root of unity, hence by Corollary 3.7 of Shorey and Tijdeman \cite{ShT} we have
  $$
  |g_k(\alpha_k)\alpha_k^m + g_k(\alpha_{k-1})\alpha_{k-1}^m| \ge |\alpha_k|^m \exp(-c_2\log|m|),
  $$
  provided $|m|\ge c_3$ with effectively computable constants $c_2,c_3>0$. This together with \eqref{eq:Binet} implies the assertion.
\end{proof}

\section{On the 4-generalized Fibonacci numbers}

In the former sections we proved non-effective and effective results on $k$-generalized Fibonacci numbers extending their definition to negative indices. Our results are far from the exactness of the above cited Theorems of Marques \cite{Marques} or Bravo, G\'omez and Luca \cite{BGL}. The reason is that inequality \eqref{eq:domineq} is much weaker than \eqref{eq:Dresden1}. The aim of this section is to show that already this weaker inequality allows us to solve completely diophantine equations related to $4$-generalized Fibonacci numbers.

\begin{theorem} \label{th:4-gen}
If $(n,m)\in \Z^2$ is a solution of the equation
\begin{equation}\label{eq:4-gen}
  F_m^{(4)} = \pm F_n^{(4)}
\end{equation}
then $\max\{|n|,|m|\} \le 22$. The exact values are given in Tables 1. and 2.
\end{theorem}

To prove this result we need a deep tool of transcendental number theory, a lower bound for linear forms in logarithms of algebraic numbers. The first such bound was proved by A. Baker \cite{Baker}, but we use here a recent variant of Matveev \cite{Matveev}, which is more convenient for the numerical investigations.

For an algebraic number $\alpha$ denotes $a_kx^k +a_{k-1}x^{k-1}+\dots+a_0\in \Z[x], a_k>0$ its minimal polynomial and by $\alpha=\alpha^{(1)}, \ldots, \alpha^{(k)}$ its conjugates. The absolute logarithmic or Weil height of $\alpha$ is
$$
h(\alpha)=\frac{1}{k}\left( \log a_k + \sum_{i=1}^{k} \log(\max\{|\alpha^{(i)}|,1\} \right).
$$
With this notation Matveev \cite{Matveev} proved

\begin{theorem} \label{th:Matveev}
  Let $\K$ be a number field of degree $k$ over $\Q$, $\gamma_1,\ldots,\gamma_t$ be positive real numbers of $\K$, and $b_1,\ldots, b_t$ rational integers. Put
$$
B \ge \max \{|b_1|,\ldots, |b_t|\}
$$
and
$$
\Lambda = \gamma_1^{b_1}\cdot\ldots\cdot\gamma_t^{b_t} - 1.
$$
Let $A_1,\ldots,A_t$ be real numbers such that
$$
A_i \ge \max\{Dh(\gamma_i), |\log \gamma_i|, 0.16\}\;  i = 1,\ldots, t.
$$
Then, assuming that $\Lambda \not= 0$, we have
$$
|\Lambda| > \exp (-1.4 \cdot 30^{t+3}\cdot t^{4.5} \cdot D^2(1 + \log D)(1 + \log B)A_1\cdots A_t).
$$
\end{theorem}

{\bf Proof of Theorem \ref{th:4-gen}}
  Until the proof of the upper bound for $|n|$ in Case iii. we follow the proof of Theorem \ref{th:effective} with $k=l=4$. We again distinguish three cases according the signs of $n$ and $m$, but investigate them in different order depending on the difficulty of the proof. In the sequel $\alpha_1,\ldots,\alpha_4$ will denote the roots of $T_4(x)$ ordered such that $\alpha_1>1> |\alpha_2| = |\alpha_3|> |\alpha_4|$. Notice that $\alpha_2$ and $\alpha_3$ are conjugate comőlex numbers and $\alpha_4<0$ is real.

  {\bf Case i:} $m,n \ge 0$ was handled in Theorem \ref{th:effective}.

  {\bf Case ii:} $m,n < 0$. This case needs only elementary consideration. As $F_0^{(4)}= F_{-1}^{(4)}= F_{-2}^{(4)}= 0$ we may assume $n< m<-1$ without loss of generality. Hence $|\alpha_4|^n>|\alpha_j|^n$, and, similarly $|\alpha_4|^m>|\alpha_j|^m$ for $j=1,2,3$. Using the Binet' formula \eqref{eq:Binet} we rewrite \eqref{eq:4-gen} in the form
  \begin{eqnarray*}
     g_4(\alpha_4) \alpha_4^{m-1} \pm  g_4(\alpha_4) \alpha_4^{n-1} &=& \sum_{i=1}^{3} \left(g_4(\alpha_i) \alpha_i^{m-1} \pm g_4(\alpha_i) \alpha_i^{n-1}\right) \\
    g_4(\alpha_4) \alpha_4^{m-1}\left( 1 \pm \alpha_4^{n-m} \right) &=& \sum_{i=1}^{3}g_4(\alpha_i) \alpha_i^{m-1}\left( 1 \pm \alpha_i^{n-m} \right).
  \end{eqnarray*}
 Dividing the last equation by $g_4(\alpha_4) \alpha_4^{m-1}$ we get
\begin{equation}\label{eq:egyenl}
   1 \pm \alpha_4^{n-m} = \sum_{i=1}^{3} \frac{g_4(\alpha_i)}{g_4(\alpha_4)} \left(\frac{\alpha_i}{\alpha_4} \right)^{m-1}\left( 1 \pm \alpha_i^{n-m} \right).
\end{equation}

\noindent
Now taking into account that $m\le -2$ thus $|\frac{\alpha_2}{\alpha_4}|^{m-1}= |\frac{\alpha_3}{\alpha_4}|^{m-1}< 0.85, |\alpha_4|<\alpha_1$, and $|\alpha_2^{n-m}|=|\alpha_3^{n-m}|>1>|\alpha_1^{n-m}|$ we get
  $$
  |\alpha_4|^{n-m} -1 \le | 1 \pm \alpha_4^{n-m} | < 1.7\left| \frac{g_4(\alpha_2)}{g_4(\alpha_4)}\right| (|\alpha_2|^{n-m}+1) + 2 \left| \frac{g_4(\alpha_1)}{g_4(\alpha_4)}\right|.
  $$
This implies after simply computation
$$
|\alpha_4|^{n-m}< 1.65 \cdot |\alpha_2|^{n-m} + 10.23,
$$
which is impossible if $n-m<-15$.

Using that $-15\le n-m\le -1$ equation \eqref{eq:egyenl} implies
\begin{eqnarray*}
0.29 \cdot|\alpha_4|^{m-1}&<& |1\pm \alpha_4^{n-m}||\alpha_4|^{m-1}=\left| \sum_{i=1}^{3} \frac{g_4(\alpha_i)}{g_4(\alpha_4)} \alpha_i^{m-1}\left( 1 \pm \alpha_i^{n-m} \right)\right|\\
&<& 41.2\cdot|\alpha_2|^{m-1} + 7.6,
\end{eqnarray*}
which is impossible if $m-1<-90$. Thus $m \ge -89$ and $n\ge -15+m\ge -104$. Computing $F_n^{(4)},\; -1\ge n\ge -104$ we obtain that $F_n^{(4)}=\pm c$ has more than one solutions only for the values given in Table 1.

\medskip
\begin{center}
  \begin{tabular}{|c||c|c|c|}
\hline
  c & 0 & 1 & 8\\
  \hline
  -n & 1,2,3,5,7,11& 4,5,10,15&13,16\\
  \hline
\end{tabular}\\[0.5ex]
Table 1
\end{center}
\medskip
{\bf Case iii:} $m\ge 0, n < 0$. Now $F_m^{(4)} = \sum_{j=1}^{4}g_4(\alpha_j)\alpha_j^m$, which is, by \eqref{eq:Dresden1} asymptotically equal to its dominating term $g_4(\alpha_1)\alpha_1^m$. For $n<0$ the roles of $\alpha_1$ and $\alpha_4$ interchange, $F_n^{(4)} = \sum_{j=1}^{4}g_4(\alpha_j)\alpha_j^n$ admits the dominating term $g_4(\alpha_4)\alpha_4^n$, but its dominance is disturbed by the fluctuation of $g_4(\alpha_2)\alpha_2^n+ g_4(\alpha_3)\alpha_3^n$. Moreover $\alpha_1$ and $\alpha_4$ are multiplicatively independent, i.e, to solve \eqref{eq:4-gen} we have to use tools of transcendental number theory.

First we actualize Lemma \ref{l:domineq} and obtain $\delta = 0.786, c_1 = 0.92$, thus
$$
|F_n^{(4)} - g_4(\alpha_4)\alpha_4^{n-1}|< 0.92\cdot |\alpha_4^{0.786 n}|.
$$
From this we easily conclude
$$
|F_n^{(4)}|> 0.193\cdot |\alpha_4|^{n} - 0.92\cdot |\alpha_4|^{0.786 n} = 0.193\cdot |\alpha_4|^{0.786 n} (|\alpha_4|^{0.214 n}- 4.767) > 0.096|\alpha_4|^{0.786n},
$$
provided $n\ge -42$. We proved in Case ii that apart the values of Table 1 all integers appear at most once in the negative $4$-generalized Fibonacci sequence. On the other hand the positive branch of the $4$-generalized Fibonacci sequence is strict monotone increasing thus it is a simple task to check that the only solutions of \eqref{eq:4-gen} with $n>-225$ are given in Table 2.
\medskip
\begin{center}
  \begin{tabular}{|c||c|c|c|c|c|}
\hline
  c & 1 & 2 & 4 & 8 & 56\\
  \hline
  m& 0,1 & 2 & 3 & 5 & 8 \\
  \hline
  -n & 4,5,10,15 & 8 & 12 & 13,16 & 22\\
  \hline
\end{tabular}\\[0.5ex]
Table 2
\end{center}

Hence in the sequel we may assume $n\le -225$, thus $|F_n^{(4)}|>0.096 \cdot |\alpha_4|^{0.786n}$. On the other hand \eqref{eq:Dresden1} with $k=4$ implies $|F_m^{(4)}|<g_4(\alpha_1)\alpha_1^{m-1}+ \frac12< 0.567 \alpha_1^{m-1}+ \frac12$. Hence, if $n,m\in \Z, n\le -50, m>0$ is a solution of \eqref{eq:4-gen} then $-n> m$.

The equality \eqref{eq:4-gen} implies
\begin{eqnarray*}
  \left|g_4(\alpha_1)\alpha_1^{m-1} \mp g_4(\alpha_4)\alpha_4^{n-1}\right| &=& \left| F_m^{(4)} - g_4(\alpha_1)\alpha_1^{m-1} \pm F_n^{(4)} \mp g_4(\alpha_4)\alpha_4^{n-1}\right| \\
   &<& 0.92\cdot |\alpha_4|^{0.786 n} + 1/2 < |\alpha_4|^{0.786 n}.
\end{eqnarray*}
We divide the last inequality by $\pm g_4(\alpha_4)\alpha_4^{n-1}$ and obtain
\begin{equation}\label{e:felso}
  \left|\pm \frac{g_4(\alpha_1)}{g_4(\alpha_4)} \alpha_1^{m-1} \alpha_4^{-n+1} -1 \right| < \left|\frac{\alpha_4}{g_4(\alpha_4)}\right| |\alpha_4|^{-0.214 n} < \exp(0.0546 n + 1.648).
\end{equation}

We apply Theorem \ref{th:Matveev} with the choices $t=3, \gamma_1 = \pm \frac{g_4(\alpha_1)}{g_4(\alpha_4)}, \gamma_2= \alpha_1, \gamma_3=\alpha_4, b_1=1, b_2=m-1,b_3=n-1$. Thus $\Lambda = \pm \frac{g_4(\alpha_1)}{g_4(\alpha_4)} \alpha_1^{m-1} \alpha_4^{-n+1}$, which is $\not=1$ by Marques \cite{Marques}, pp. 460, 461. Plainly $\K=\Q(\alpha_1,\alpha_4),$ hence $D\le 12$, and as $-n>m\ge 0$ we have $B=|n|+1$. We still have to establish $A_1,A_2,A_3$. Clearly $h(\gamma_2)=h(\gamma_3)=(\log(\alpha_1)-\log(-\alpha_4))/4<0.228$, thus $A_2=A_3=2.736$ is an allowed choice. Again by Marques \cite{Marques} (p. 460) we have $h(\gamma_1)\le 4 \log 3 $, thus $A_1= 52.74 > 48 \log 3$ is a correct choice. Summarizing we get
\begin{equation}\label{e:also}
  |\Lambda-1|> - 2.84 \cdot 10^{16} \log|n|.
\end{equation}

Comparing \eqref{e:also} and \eqref{e:felso} we get the inequality
$$
0.0546 |n| - 1.648 < 2.84 \cdot 10^{16} \log|n|,
$$
which is equivalent to
$$
|n| < 5.2015\cdot 10^{17} \log|n| + 30.1832.
$$
A straightforward computation shows that this inequality is impossible, if $|n|> 2.32\cdot 10^{19}$.\\[1ex]

{\bf Reduction of the bound for $|n|.$} The just proved upper bound for $|n|$ is huge, but, fortunately, the Baker-Davenport method \cite{Baker_Davenport} enables us to reduce it considerable. Since its discovery fifty years ago it was used plenty of times and was successful in all known cases. As $n\le -225$ we have $|\Lambda -1|< 0.338$, thus
$$
\left|(n-1)\log|\alpha_4| - (m-1)\log \alpha_1 - \log\frac{g_4(\alpha_1)}{g_4(\alpha_4)}\right| = \log|\Lambda|<6.33 \cdot |\alpha_4|^{-0.214 n}
$$
by Lemma 2.2 of de Weger \cite{deWeger}. After division by $\log \alpha_1$ we get
\begin{equation}\label{e:redhato}
  \left|(n-1)\frac{\log|\alpha_4|}{\log \alpha_1} - (m-1)- \log\frac{g_4(\alpha_1)}{g_4(\alpha_4)}/\log\alpha_1 \right| < 9.65 \cdot |\alpha_4|^{-0.214 n}.
\end{equation}
We apply to this inequality the first assertion of Lemma 5 of Dujella and Peth\H{o} \cite{Dujella_Petho} in slightly modified form

\begin{lemma} \label{l:DP}
  Suppose that $\kappa,\mu\in \R$, and $A,B,M>0$ with $M\in \Z$.  Let $p/q$ be the convergent of the continued fraction expansion of $\kappa$ such that $q >M$ and let $\varepsilon =\| \mu q\| - M\cdot \|\kappa q\|$, where $\|.\|$ denotes the distance from the nearest integer. If $\varepsilon >0$, then there is no solution of the inequality
  $$
   |n\kappa - m +\mu| < AB^{-n}
  $$
  in integers $m$ and $n$ with
  $$
  \frac{\log(Aq/\varepsilon)}{\log B}  \le n \le M .
  $$
\end{lemma}

\begin{proof}
  The same as the proof of Lemma 5 of \cite{Dujella_Petho}.
\end{proof}
We apply Lemma \ref{l:DP} with the straightforward choices
$$
\kappa = -\frac{\log|\alpha_4|}{\log \alpha_1},\quad \mu = - \log\frac{g_4(\alpha_1)}{g_4(\alpha_4)}/\log\alpha_1,
$$
$$
M= 10^{20}>|n|+1,\quad B=|\alpha_4|^{-0.214} = 1.056, \quad A= 9.65/B= 9.14.
$$ Notice that $-n+1=|n|+1\ge 226$.

We have to compute the continued fraction expansion of $\kappa$ so far that the denominator of the last, say $k_0$-th, convergent is larger than $M$. The denominator of the $k$-th convergent of any continued fraction is bounded below by $\left(\frac{1+\sqrt{5}}{2}\right)^k$, thus we may take $k_0=95$.

For safety we performed all computation with $100$ decimal digits precision. Already the denominator of the $45$th convergent - $66618036593827352256020$ - was larger than $M$. As the corresponding $\varepsilon>0.04$ we conclude that \eqref{e:redhato} is impossible if $|n|> 1063$. We repeat the above computation, but this time with M=1064. Now already the denominator of the $8$th convergent - 3336 - is larger than $1064$, and as the corresponding $\varepsilon>0.136$ we obtain the new bound $|n|\le 225$. At the beginning of the proof we tested that \eqref{eq:4-gen} has no solution below this bound, hence the proof is complete.
$\Box$\\[1ex]

\medskip
{\bf Acknowledgement} I thank L\'aszl\'o Szalay for his valuable remarks, especially pointing out an error, on an earlier version of this paper.

\end{document}